\documentclass{article}

\usepackage{amsmath, amssymb, amsthm}
\usepackage{graphicx}
\usepackage{epstopdf}
\usepackage{color}

\newtheorem*{defn}{Definition}
\newtheorem{lemma}{Lemma}
\newtheorem{cor}{Corollary}

\newtheorem{theorem}{Theorem}

\newtheorem{prop}{Proposition}

\newtheorem*{lemma*}{Lemma}
\newtheorem*{thrm*}{Theorem}

\theoremstyle{definition}

\newtheorem*{remark}{Remark}

\newtheorem*{ack}{Acknowledgements}

\newcommand{\cinf}{C^\infty}
\newcommand{\ccinf}{C_c^\infty}

\newcommand{\cs}{\mathcal{E}^\prime}
\newcommand{\ds}{\mathcal{D}^\prime}

\newcommand{\supp}{\mathop{\mathrm{supp}}}

\newcommand{\Vol}{\mathrm{Vol}}

\newcommand{\WF}{\mathop{\mathrm{WF}}}

\renewcommand{\Im}{\mathop{\mathrm{Im}}}

\newcommand{\C}{\mathbb{C}}
\newcommand{\N}{\mathbb{N}}
\newcommand{\R}{\mathbb{R}}

\title{Injectivity and stability for a generic class of \\ generalized Radon transforms}
\author{Andrew Homan and Hanming Zhou}
\date{}

\begin{document}
\maketitle

\begin{abstract}
Let $(M, g)$ be an analytic, compact, Riemannian manifold with boundary, of dimension $n \ge 2$. We study a class of generalized Radon transforms, integrating over a family of hypersurfaces embedded in $M$, satisfying the Bolker condition \cite{quinto94}. Using analytic microlocal analysis, we prove a microlocal regularity theorem for generalized Radon transforms on analytic manifolds defined on an analytic family of hypersurfaces. We then show injectivity and stability for an open, dense subset of smooth generalized Radon transforms satisfying the Bolker condition, including the analytic ones.
\end{abstract}

\section{Introduction}

Let $(M, g)$ be an analytic, compact, Riemannian manifold with boundary, of dimension $n \ge 2$, with volume form denoted by $d\Vol$. Let $\Sigma$ be a family of embedded hypersurfaces. A generalized Radon transform takes each $f \in \cinf(M)$ to the set of its integrals over the hypersurfaces of $\Sigma$, with respect to the surface measure induced by the volume form. Often $\Sigma$ is itself a smooth manifold; for example, the Euclidean Radon transform is defined over the set of affine hypersurfaces in $\R^n$, which has $\R \times S^{n-1}$ as a double cover. Such transforms are found in applications to many other fields, including harmonic analysis, scattering theory, seismology and medical imaging.

The main questions regarding these transforms include determining conditions under which they are injective, finding when the transform has a stable inversion, and characterizing the range. We concentrate on the first two questions here. These questions are also important in the partial data case, where integrals are known only for a subset of $\Sigma$. In the case of the Euclidean Radon transform, we refer the reader to \cite{helgason99, natterer01, ehrenpreis03, helgason11} and references therein for the resolution of these problems, and their generalization to the Radon transform over symmetric spaces and other related contexts.

The geometric data of a generalized Radon transform can be encoded by an incidence relation between points on $M$ and the hypersurfaces in $\Sigma$ that contain them. Let $\Lambda \subset M \times \Sigma$ be this relation, i.e., the set of ordered pairs $(x, \sigma)$ such that $x \in \sigma$. One says $\Lambda$ is a double fibration when it is a smooth, embedded submanifold of $M \times \Sigma$ such that both canonical projections are smooth and their restrictions to $\Lambda$ form a fiber bundle over $M$ and $\Sigma$, respectively \cite{gelfand69}. 

Guillemin and Sternberg \cite{guillemin79, guillemin85, guillemin90} showed that given a Radon transform $R$ defined by a double fibration $\Lambda$, both $R$ and its adjoint $R^*$ (often called the generalized backprojection by analogy with the Euclidean case) are Fourier integral operators, and the canonical relation of $R$ is the conormal bundle $N^*\Lambda$ of the incidence relation. If in addition $\Lambda$ satisfies the Bolker condition, which says that the induced projection $\pi^*: N^*\Lambda \to T^*\Sigma$ is an embedding, then the normal operator $R^*R$ is an elliptic pseudodifferential operator, which yields invertibility up to smoothing error.

A stronger result than invertibility is a Helgason-type support theorem, by analogy with that of the Euclidean Radon transform \cite{helgason80}. Such a support theorem implies that if $f$ is \emph{a priori} of compact support, and $Rf = 0$ for all hypersurfaces intersecting the support, then $f = 0$. In the analytic category, there has been much work in this area (for example, \cite{boman87, boman93, quinto94, quinto06}) using analytic microlocal analysis and the Bolker condition to prove support theorems for analytic generalized Radon transforms (i.e., with $M, \Sigma$ analytic manifolds) with nonvanishing analytic weight. On the other hand, in $n = 2$ there is a counterexample in the smooth category due to Boman \cite{boman93a} of a function supported in the disk such that some weighted Radon transform over lines vanishes. For the weighted X-ray transform over curves, there is an analogous Bolker condition \cite{greenleaf90} and support theorems are known for a class of such transforms in $n \ge 3$ \cite{uhlmann12, zhou13} including the geodesic ray transform on an analytic, simple manifold over functions \cite{krishnan09} and over symmetric tensor fields \cite{krishnan09a}.

Our first result considers such analytic generalized Radon transforms and shows their analytic microlocal regularity. This builds upon similar results for the weighted X-ray transform over a generic class of curves \cite{dairbekov07, frigyik08} and in particular, geodesics \cite{stefanov04, stefanov05}. To avoid complications at the boundary of $M$, we embed it isometrically in a slightly larger, open analytic manifold $M_1$. In section 2, we show how to extend the definition of $R_w$ to a transform on $M_1$ in a stable way. We show:

\begin{theorem}
\label{microlocal-regularity}
Let $R_w$ be an analytic generalized Radon transform satisfying the Bolker condition, with $w$ an analytic, nonvanishing weight. Let $f \in \cs(M_1)$ be such that $R_wf(\sigma) = 0$ in a neighborhood of some hypersurface $\sigma_0 \in \Sigma$. Then the analytic wavefront set $\WF_A(f)$ does not intersect the conormal bundle $N^*\sigma_0$.
\end{theorem}

The main tool is a complex stationary phase lemma of Sj\"ostrand \cite[Theorem 2.8, 2.10 ff.]{sjostrand82} and related techniques, which suffice in lieu of a hypothetical analytic calculus of Fourier integral operators. The proof of this theorem is given is section 3.

Theorem \ref{microlocal-regularity} also implies a local support theorem \cite[Prop 2.3]{quinto94} and in particular the injectivity of $R_w : L^2(M) \to H^{(n-1)/2}(\Sigma)$. In fact, the proof of the theorem shows that it suffices for $R_wf(\sigma)$ to be analytic in a neighborhood of $\sigma_0$. One then obtains a unique continuation result of the following type: if $f \in \cs(M_1)$ is analytic on one side of $\sigma_0$, and $R_wf(\sigma)$ is analytic in a neighborhood of $\sigma_0$, then there is a neighborhood of $\sigma_0$ on which $f$ is analytic.

Our second result is a stability estimate for a generic class of smooth generalized Radon transforms satisfying the Bolker condition. We restrict ourselves to those generalized Radon transforms studied by Beylkin \cite{beylkin84}, which have $\Sigma$ parametrized globally by the level sets of a smooth defining function $\varphi$ satisfying some conditions to be made explicit later.

\begin{theorem}
\label{stability-perturbation}
Let $(M, g)$ be an analytic Riemannian manifold with boundary. Take $R_w : L^2(M) \to H^{(n-1)/2}(\Sigma)$ to be an injective generalized Radon transform defined by $\varphi$ with weight $w$, satisfying the Bolker condition. Then there exists $K \gg n$ and a neighborhood of $(\varphi, w) \in C^K$ such that the generalized Radon transform $\tilde{R}_{\tilde{w}}$ defined on $(M, g)$ by a defining function and weight in this neighborhood is injective and for all $f \in L^2(M)$ there exists $C > 0$ such that
\[
  ||f||_{L^2(M)} \le C|| \tilde{R}_{\tilde{w}}^*\tilde{R}_{\tilde{w}}^{\vphantom{*}} f||_{H^{n-1}(M_1)}.
\]
\end{theorem}

This follows from an analysis of the symbol of the normal operator $R_w^*R_w^{\vphantom{*}}$. As mentioned above, under the Bolker condition it is an elliptic pseudodifferential operator. We show that perturbing the defining function and weight slightly in $C^K$ perturbs the operator slightly, preserving the stability estimate.

While we work entirely on an analytic Riemannian manifold $(M, g)$ and do not perturb the metric in this result, we use the metric only to provide a convenient choice of surface measure, and to ensure the existence of a dense set of injective, generalized Radon transforms. We may then conclude:

\begin{cor}
On each analytic, compact Riemannian manifold with boundary, there is a generic set of generalized Radon transforms satisfying the Bolker condition that are both injective and stable.
\end{cor}

We defer the proof of Theorem \ref{stability-perturbation} to section 4.

\begin{ack}
The first author is partly supported by NSF Grant DMS--1301646.
\end{ack}

\section{Generalized Radon Transforms}

In this section we fix notation and establish some basic facts about the generalized Radon transform, including a statement of the Bolker condition. For concreteness we consider the space of hypersurfaces $\Sigma$ as parameterized by a defining function, following Beylkin \cite{beylkin84}, though we only consider oriented hypersurfaces. To avoid difficulties occuring at the boundary of $M$, we assume $M$ is isometrically embedded in a slightly larger open manifold $M_1$, whose metric we also refer to by $g$. If we are considering the analytic category of Radon transforms, we will also assume $M_1$ is analytic. In the sequel, we will always consider $L^2(M)$ to be functions on $M_1$, extended by zero.

\begin{defn}
Let $\varphi \in \cinf(M_1 \times (\R^n\setminus 0))$. $\varphi$ is a \emph{defining function} when it satisfies the following conditions:
\begin{enumerate}
\item $\varphi(x, \theta)$ is positive homogeneous of degree one in the fiber variable.
\item $\varphi$ is non-degenerate in the sense that $d_x\varphi(x, \theta) \not= 0$.
\item The mixed Hessian of $\varphi$ is strictly positive, i.e.,
\[
 \det\left( \frac{\partial^2 \varphi}{\partial x^i \partial \theta^j} \right) > 0.
\]
\end{enumerate}
The level sets of $\varphi$ will be denoted by
\[
 H_{s, \theta} = \{ x \in M_1 : \varphi(x, \theta) = s \}.
\]
\end{defn}
Note that by homogeneity, $H_{s, \theta} = H_{\lambda s, \lambda \theta}$ for $\lambda > 0$. Therefore we can consider $\Sigma$ as globally parameterized by $(s, \theta) \in \R \times S^{n-1}$. Often we will also implicitly consider $\varphi$ as a function on $M_1 \times S^{n-1}$.

The third condition imposed on a defining function is a local form of Bolker's condition. This allows us to locally identify $(x, \theta) \in M_1 \times S^{n-1}$ with the covector $d_x\varphi(x, \theta)/|d_x\varphi(x, \theta)|_g \in S_x^*M_1$. We will assume in addition a stronger, global Bolker condition. 
\begin{defn}
A defining function $\varphi$ satisfies the \emph{global Bolker condition} if for each $\theta \in S^{n-1}$, the map $x \mapsto d_\theta\varphi(x, \theta)$ is injective, and for each $x \in M$, the map $\theta \mapsto d_x\varphi(x, \theta)$ is surjective.
\end{defn}
The first condition is roughly analogous to the ``no conjugate points'' condition assumed by \cite{frigyik08, krishnan09} for similar results regarding the geodesic ray transform, and the second ensures that every singularity is observable from some hypersurface in $\Sigma$. Note that generalized Radon transforms defined by a double fibration satisfying the Bolker condition as stated by Guillemin et.~al.~also satisfy this Bolker condition, see \cite[Lemma 3.5]{quinto94}.

Consider $f \in \cinf(M)$. We extend it by zero to a function on $M_1$ which we also denote by $f$. Let the generalized Radon transform $R_w$ determined by $(M, g, \varphi, w)$ be defined by
\[
 R_w f(s, \theta) = \int_{H_{s, \theta}} w(x, \theta) f(x)\, d\mu_{s, \theta},
\]
where $w \in \cinf(M_1 \times S^{n-1})$ is a smooth, nonvanishing weight and $d\mu_{s, \theta}$ is the volume form on $H_{s, \theta}$ induced by $d\Vol$. There exists a smooth, nonvanishing function $J(x, \theta)$ such that
\[
 d\mu_{s, \theta}(x) \wedge ds = J(x, \theta)\, d\Vol.
\]
We calculate the adjoint of $R_w$ in $L^2(M, d\Vol)$ to be
\begin{align*}
\int_{S^{n-1}} \int_\R (R_w f) \overline{g}\, ds\, d\theta &= \int_{S^{n-1}} \int_\R \int_{H_{s, \theta}} w(x, \theta) f(x)\overline{g}(s, \theta)\, d\mu_{s, \theta} \, ds\, d\theta \\
&= \int_{S^{n-1}} \int_{M_1} \overline{g}(\varphi(x, \theta), \theta) w(x, \theta) J(x, \theta) f(x)\, d\Vol\, d\theta
\end{align*}
Therefore
\[
 R^*_w g(x) = \int_{S^{n-1}} \overline{w}(x, \theta) \overline{J}(x, \theta) g(\varphi(x, \theta), \theta)\, d\theta.
\]
This is simply a generalized backprojection with weight $\overline{wJ}$.

\section{Microlocal regularity}

In this section, we take $(M_1, g)$ to be an analytic Riemannian manifold, $\varphi$ to be an analytic defining function, and $w$ to be an analytic nowhere vanishing weight. Given $f \in \cs(M_1)$, we are interested in the microlocal analyticity of $f$ given that of $R_w f$. (We extend $R_w$ to $\cs(M_1)$ by duality.) We will use the following definition of the analytic wavefront set, following Sj\"ostrand. There are alternative approaches to analytic wavefront set by Sato, Kawai, Kashiwara \cite{sato73} and also Bros and Iagolnitzer \cite{bros75}, which were shown to be equivalent by Bony \cite{bony76}.

\begin{defn}[{\cite[Def.\ 6.1]{sjostrand82}}]
Let $(x_0, \xi_0) \in T^*\R^n \setminus 0$ and let $\psi(x, y, \xi)$ be an analytic function defined in a neighborhood $U$ of $(x_0, x_0, \xi_0) \in \C^{3n}$ such that
\begin{enumerate}
\item For all $(x, x, \xi) \in U$ (i.e., $x=y$), we have 
\[
\psi(x, x, \xi) = 0 \text{\ and\ } \partial_x\psi(x, x, \xi) = \xi.
\]
\item There exists $C > 0$ such that for all $(x, y, \xi) \in U$, we have 
\[
\Im \psi(x, y, \xi) \ge C|x - y|^2.
\]
\end{enumerate}
Let $a(x, y, \xi)$ be an elliptic classical analytic symbol defined on $U$, see, e.g., \cite[Theorem 1.5]{sjostrand82}.

We say $u \in \ds(\R^n)$ is analytic microlocally near $(x_0, \xi_0)$ if there exists a cut-off function $\chi \in \ccinf(\R^n)$ with $\chi(x_0) = 1$ such that
\[
\int e^{i\lambda\psi(x, y, \xi)}a(x, y, \xi)\chi(y)\overline{u(y)}\, dy = O(e^{-\lambda/C}),
\]
for some $C > 0$, uniformly in a conic neighborhood of $(x_0, \xi_0)$.

The analytic wavefront set is the closed conic set $\WF_A(u) \subset T^*\R^n\setminus 0$, which is the complement of the set of covectors near which $u$ is microlocally analytic.
\end{defn}
We note that this definition is microlocal and invariantly defined, and therefore can be extended to distributions on analytic manifolds (see \cite[Theorem 8.5.1]{hormander-i} and the remarks following). In this case, for $u \in \ds(M_1)$, $\WF_A(u)$ is a closed conic subset of $T^*M_1 \setminus 0$.

Recall that since the mixed Hessian of $\varphi$ is strictly positive, we may locally identify $(x, \theta) \in M_1 \times S^{n-1}$ with the unit covector $d_x\varphi(x, \theta)/|d_x\varphi(x, \theta)|_g \in S^*M_1$. Fix a covector $(x_0, \theta_0) \in T^*M_1 \setminus 0$ with $s_0 = \varphi(x_0, \theta_0)$. From now on we will work in a small conic neighborhood of this covector.

\begin{prop}
If $R_wf(s, \theta) = 0$ for $(s, \theta)$ in a neighborhood of $(s_0, \theta_0)$, then $(x_0, d_x\varphi(x_0, \theta_0)) \not\in \WF_A(f)$.
\end{prop}

\begin{proof}
Let us fix a coordinate system. We already have local coordinates $(x, \theta)$ on $T^*M \setminus 0$. Without loss of generality we can take $s_0 = 0$ and $|\theta_0| = 1$. To simplify the coordinates on $\Sigma$, we perform a stereographic projection onto the tangent plane of the sphere at $\theta_0$, which is an analytic diffeomorphism mapping a neighborhood of $\theta_0 \in S^{n-1}$ to a neighborhood of the origin in $\R^{n-1}$. We refer to the coordinates on this tangent plane by $\xi$, and pass to a perhaps smaller neighborhood of $\Sigma$ with $|s| < 2\epsilon$ and $|\xi| < \delta$, with $\epsilon, \delta > 0$ being small parameters.

Much of the complexity of analytic microlocal calculus is due to the difficulty of localizing in the analytic category, as there are no suitable cut-off functions. Instead one often uses a sequence of quasianalytic cut-off functions $\chi_N \in \ccinf(\R)$, depending on $\epsilon$, for whose construction we refer to \cite{hormander71, treves80}. We will only use the following properties of this sequence:
\begin{enumerate}
\item $\supp \chi_N \subset (-2\epsilon, 2\epsilon)$ and $\chi_N(-\epsilon, \epsilon) = 1$.
\item For all $N \in \N$ and $k \le N$, the estimate
\[
\left| \partial^{(k)}_s \chi_N(s) \right| \le (CN)^k
\]
holds for a constant $C > 0$ independent of $N$.
\end{enumerate}

By assumption $R_w f(s, \xi) = 0$ for $|s| < 2\epsilon$ and $|\xi| < \delta$. Let $\lambda \gg 1$ be a large parameter, to be fixed later. This implies that
\begin{equation}
\label{local}
0 = \int e^{i\lambda s} \chi_N(s) \int_{H_{s, \xi}} w(x, \xi)f(x)\, d\mu_{s, \xi}\, ds.
\end{equation}
Recall that $\xi$ are analytic coordinates for the neighborhood of $\theta_0$ in $S^{n-1}$ that we are concerned with, and so here and in the sequel we write for brevity, e.g., $w(x, \xi) = w(x, \theta(\xi))$ and $d\mu_{s, \xi} = d\mu_{s, \theta(\xi)}$.

It follows from Beylkin's construction that
\[
 d\mu_{s, \xi}\wedge ds = J(x, \xi)\, d\mathrm{Vol}
\]
where $J(x, \xi)$ is an analytic, nonvanishing Jacobian and $d\mathrm{Vol}$ is the volume form on $M_1$ associated to the metric. Hence \eqref{local} reduces to the oscillating integral
\begin{equation}
\label{fio}
\int e^{i\lambda\varphi(x, \xi)} a_N(x, \xi)f(x)\, d\mathrm{Vol} = 0.
\end{equation}
Here $a_N(x, \xi)$ is a sequence of classical analytic symbols on the same neighborhood of $(x_0, 0) \in M_1 \times \R^{n-1}$. The coordinates on $x$ and $\xi$ are real-analytic, and so we may extend their domain of definition slightly by analytic continuation to a Grauert tube of a small neighborhood of $H_{0, 0} \subset M_1$ (for x) and a small neighborhood of the origin in $\C^{n-1}$ for $\xi$. This continuation in principle depends on the choice of analytic coordinates, but as the analytic wavefront set is invariantly defined the final result does not depend on this choice. We choose a perhaps smaller $\delta$ such that $\{\xi \in \C^{n-1} : |\xi| < \delta/2\}$ is contained in this neighborhood. We denote the local complex coordinate patch of $x_0$ as $U \subset \C^n$.

Let $y \in U$ and $\eta \in \C^{n-1}$, with $|\eta| < \delta/2$. Let $\rho(\xi) = 1$ when $|\xi| \le \delta$ and zero otherwise. Then we multiply \eqref{fio} by
\[
 \rho(\xi - \eta) \exp\left(-\frac{\lambda}{2} |\xi - \eta|^2 - i\lambda \varphi(y, \xi)\right),
\]
and integrate with respect to $\xi$. The resulting integral is of the form
\begin{equation}
\label{augment}
\iint e^{i\lambda\Phi(x, y, \xi, \eta)} b_N(x, \xi, \eta) f(x)\, d\mathrm{Vol}(x)\, d\xi = 0.
\end{equation}
Here $b_N$ is a sequence of classical analytic symbols defined on a complex neighborhood of $H_{0, 0} \times \{0\} \times \{0\}$ and $\Phi$ is the augmented phase function given by
\[
 \Phi(x, y, \xi, \eta) = \frac{i}{2}|\xi-\eta|^2 + \varphi(x, \xi) - \varphi(y, \xi).
\]

To estimate the left-hand side of \eqref{augment}, we intend to use the method of complex stationary phase. Therefore, we are interested in the critical points of the function $\xi \mapsto \Phi(x, y, \xi, \eta)$. Note that
\[
 \Phi_\xi(x, y, \xi, \eta) = i(\xi-\eta) + \partial_\xi\varphi(x, \xi) - \partial_\xi\varphi(y, \xi).
\]
There are clearly real critical points $\xi$ when $\xi = \eta$ and $x = y$. These critical points are non-degenerate, and therefore induce complex critical points $\xi_c(x, y, \eta) = \eta + i(y-x) + O(\delta)$.

Consider the situation when $y = 0$. Then for $x \not= 0$, the only real critical points are where $\partial_\xi\varphi(x, \xi) = \partial_\xi\varphi(y, \xi)$. However, this cannot happen by the global Bolker condition that we imposed on the defining function. By non-degeneracy again we see there are no real or complex critical points other than $\xi_c(x, y, \eta)$ for $(x, y, \xi, \eta)$ where $|y| < \delta$ and $|\xi - \eta| < \delta$.

Now we apply the complex stationary phase lemma \cite[Theorem 2.8, 2.10]{sjostrand82} to \eqref{augment}. As a preparatory step divide the integral into two regions; one over the region
\[
I_+ = \{(x, y, \xi, \eta) : |x-y| \le \delta/C_0, |\xi - \eta| < \delta\}
\] and one over the region 
\[
I_- = \{(x, y, \xi, \eta) : |x-y| > \delta/C_0, |\xi - \eta| < \delta\}.
\] 
Here $C_0 > 0$ is a constant chosen so that the critical points $\xi_c(x, y, \eta)$ lie within $I_+$ and none lie in $I_-$.

In $I_-$, we may define the usual operator $L$ such that $Le^{i\lambda\Phi} = e^{i\lambda\Phi}$ via
\[
 L = \frac{\partial_\xi \overline{\Phi} \cdot \partial_\xi}{i\lambda|\partial_\xi \Phi|^2}.
\]
This is well-defined as there are no critical points in $I_-$, so we may repeatedly integrate by parts:
\begin{align*}
\left| \int_{I_-} e^{i\lambda\Phi}b_N f\, d\mathrm{Vol}\, d\xi \right| &= \left| \int_{I_-} (L^N e^{i\lambda\Phi}) b_N f\, d\mathrm{Vol}\, d\xi \right| \\
 &\le \left| \int_{I_-} e^{i\lambda\Phi} (L^*)^N[b_N f]\, d\mathrm{Vol}\, d\xi \right| + \sum_{k=1}^{N} | \mathcal{B}_k |.
\end{align*}
The terms $\mathcal{B}_k$ are boundary terms that decay exponentially, due to the fact that $\Im\Phi > 0$ for $|\xi - \eta| = O(\delta)$. As for the integral on the right-hand side, we recall that $b_N$ is defined by
\begin{equation}
\label{symbol}
 b_N(x, \xi, \eta) = \rho(\xi-\eta)\chi_N(\varphi(x, \xi))w(x, \xi)J(x, \xi).
\end{equation}
The worst possible growth of $(L^*)^Nb_N$ in terms of $N$ occurs when all derivatives are applied to $\chi_N(\varphi(x, \xi))$, and in this case we may apply the estimate
\[
 \left| \partial^{(N)}_s\chi_N(s) \right| \le (CN)^N,
\]
which follows from the construction of the sequence of quasianalytic cut-off functions. Therefore,
\begin{equation}
\label{outside-estimate}
\left| \int_{I_-} e^{i\lambda\Phi}b_N f\, dx \right| = O\left( (CN/\lambda)^N + CNe^{-\lambda/C}\right).
\end{equation}

As for the integral over $I_+$, the cut-off functions $\chi_N(\varphi(x, \xi))$ are all equal to one. Therefore the amplitude on $I_+$ does not depend on $N$; we remove this dependence and refer to the amplitude restricted to this region as $b$. We know all of the critical points of $\xi \mapsto \Phi$ and can therefore apply the complex stationary phase lemma. This yields an estimate of the form
\[
\int e^{i\lambda\Phi} b f\, d\mathrm{Vol}\, d\xi = C\lambda^{-n/2}\int e^{i\lambda\psi} B f\, d\mathrm{Vol} + O\left( (CN/\lambda)^N + Ne^{-\lambda/C}\right).
\]
Here $\psi(x, y, \eta) = \Phi(x, y, \xi_c(x, y, \eta), \eta)$ and $B(x, y, \eta) = b(x, \xi_c(x, y, \eta), \eta)$. We may now fix $N$ such that $N \le (\lambda/Ce) \le N + 1$ to ensure the error is exponentially small.
\[
 \int e^{i\lambda\psi} B f\, d\mathrm{Vol} = O(e^{-\lambda/C}).
\]
Now $B(x, y, \eta)$ is an elliptic analytic symbol near $(x, y, \eta) = 0$ and $\psi(x, y, \eta)$ is a non-degenerate phase function. To show this implies $(x_0, \theta_0) \not\in \WF_A(f)$, we check the details of the characterization of the analytic wave front set given above. Recall
\[
 \psi(x, y, \eta) = \frac{i}{2}|\xi_c(x, y, \eta) - \eta|^2 + \varphi(x, \xi_c(x, y, \eta)) - \varphi(y, \xi_c(x, y, \eta)).
\]
Note that $\xi_c(x, x, \eta) = \eta$ for $x$ real, and therefore $\psi(x, x, \eta) = 0$. In addition
\[
\partial_x \psi(x, x, \eta) = \partial_x\varphi(x, \eta) = -\partial_y \psi(x, x, \eta).
\]
By the global Bolker condition we can make a change of variables $\eta^\prime$ so that $\eta^\prime = d_x\varphi(x, \eta)$. Finally, it is clear that $\Im\psi(x, y, \xi) \ge C|x-y|^2$ for $x, y$ real. Therefore, $(x_0, d_x\varphi(x_0, \theta_0)) \not\in \WF_A(f)$.
\end{proof}

Theorem \ref{microlocal-regularity} follows from applying the proposition to all conormals of a fixed hypersurface $\sigma_0$.

\begin{remark}
From the proof we see that it suffices for $R_wf(\sigma)$ to be analytic in a neighborhood of $(s_0, \theta_0)$. After microlocalization, the right-hand side of \eqref{fio} will be $O(e^{-\lambda/C})$ instead of zero, but this poses no problem.
\end{remark}

\section{Stability}

We now return to generalized Radon transforms with smooth defining function $\varphi : M_1 \times S^{n-1}$ and smooth, nonvanishing weight $w : M_1 \times S^{n-1}$. The object of interest in this section is the normal operator $N_w = R_w^* R_w$. It is known that the global Bolker condition implies $N_w$ is a pseudodifferential operator \cite[Prop 8.2]{guillemin79}. However, we require more detailed knowledge of the symbol of $N_w$ for the kind of stability estimates we prove later.

First we obtain a representation of the Schwartz kernel of $R_w$.

\begin{lemma}
The Schwartz kernel $K_{R_w} \in \ds(\R \times S^{n-1} \times M_1)$ of $R_w$ is
\[
K_{R_w}(s, \theta, y) = (2\pi)^{-1}\delta(s - \varphi(y, \theta)) w(y, \theta) J(y, \theta)
\]
where $J(y, \theta)$ is the smooth, nonvanishing function such that
\[
 d\mu_{s, \theta}(y) \wedge ds = J(y, \theta)\, d\mathrm{Vol}(y).
\]
\end{lemma}

\begin{proof}
We perform a partial Fourier transform of $R_wf(s, \theta)$ in the $s$ variable, taking $s^\prime$ to be the dual variable of $s$. The change of variables then yields
\begin{align*}
 \mathcal{F}_sR_wf(s^\prime, \theta) &= \int_\R e^{-iss^\prime} \int_{H_{s, \theta}} w(y, \theta)f(y)\, d\mu_{s, \theta}\, ds \\
  &= \int_{M_1} e^{-is^\prime\varphi(y, \theta)} w(y, \theta) J(y, \theta) f(y)\, d\mathrm{Vol}(y).
\end{align*}
Therefore
\begin{align*}
R_wf(s, \theta) &= (2\pi)^{-1} \int_\R \int_{M_1} e^{i(s-\varphi(y, \theta))s^\prime} w(y, \theta) J(y, \theta) f(y)\, d\mathrm{Vol}(y)\, ds^\prime \\
 &= \int_{M_1} K_{R_w}(s, \theta, y) f(y)\, d\mathrm{Vol}(y). \qedhere
\end{align*}
\end{proof}

Similarly, the kernel of the generalized backprojection $R_w^\prime$ is
\[
K_{R^*_w} = (2\pi)^{-1} \delta(\varphi(x, \theta)-s) \overline{w}(x, \theta)\overline{J}(x, \theta).
\]
From this we see that the kernel of $N_w$ is
\begin{equation}
\label{normal-kernel}
K_{N_w} = (2\pi)^{-1} \iint e^{is^\prime(\varphi(x, \theta) - \varphi(y, \theta))} \overline{w}(x, \theta) \overline{J}(x, \theta) w(y, \theta) J(y, \theta) \, ds^\prime\, d\theta.
\end{equation}
We can now use this representation to find the principal symbol of the normal operator $N_w$.

\begin{lemma}
The principal symbol of $N_w$ is
\[
 p(x, \xi) = (2\pi)^{1-n}\frac{W(x, x, \xi/|\xi|) + W(x, x, -\xi/|\xi|)}{|\xi|^{n-1}},
\]
where $W$ is the auxillary function
\[
 W(x, y, \theta) = \overline{w}(x, \theta)\overline{J}(x, \theta)w(y, \theta)J(y, \theta).
\]
\end{lemma}

\begin{proof}
Beginning from \eqref{normal-kernel}, we split the integration over $\R$ into $\{s^\prime > 0\}$ and $\{s^\prime < 0\}$. Using the positive homogeneity of the defining function, we rewrite the integral as
\begin{align*}
K_{N_w} &= \int_{S^{n-1}}\int_0^\infty e^{i(\varphi(x, s^\prime\theta) - \varphi(y, s^\prime\theta))} W(x, y, \theta)\, ds^\prime \, d\theta \\
 &+ \int_{S^{n-1}}\int_0^\infty e^{-i(\varphi(x, s^\prime\theta) - \varphi(y, s^\prime\theta))} W(x, y, \theta) \, ds^\prime\, d\theta. \\
 &= K_{N_w}^+ + K_{N_w}^-.
\end{align*}
Here $K_{N_w}^+$ and $K_{N_w}^-$ are the Schwartz kernels of the operators $N_w^+$ and $N_w^-$ respectively, so that $N_w = N_w^+ + N_w^-$. We work with each term separately. Let $\xi = s^\prime\theta$ be polar coordinates for $\R^n$. This change of variables is justified when the kernel is applied to a test function in $\ccinf(M_1)$; using the proof of \cite[Theorem 7.8.2]{hormander-i} it can be shown that it is justified for the kernel itself. Then we obtain
\[
K_{N_w}^+ = \int_{\R^n} e^{i(\varphi(x, \xi) - \varphi(y, \xi))} W\left(x, y, \frac{\xi}{|\xi|}\right) |\xi|^{1-n} \, d\xi
\]
By the global Bolker condition, $\partial_\xi\varphi(x, \xi) = \partial_\xi\varphi(y, \xi)$ implies $x = y$. A stationary phase argument implies that $K_{N_w}^+$ is a smooth function away from the diagonal of $M_1 \times M_1$.

Fix $x_0 \in M_1$. There exists a neighborhood $U$ of $x_0$ on which we have normal coordinates, which we refer to again with $(x^i)$, such that $x(x_0) = 0$. We then use $(x^i, y^i)$ as coordinates on $U \times U$, with $x^i = y^i$. We consider the localized kernel
\[
 \chi K_{N_w}^+\chi = \int_{\R^n} e^{i(\varphi(x, \xi) - \varphi(y, \xi))} W\left(x, y, \frac{\xi}{|\xi|}\right) \chi(x)\chi(y) |\xi|^{1-n} \, d\xi.
\]
In these local coordinates, we can expand the phase function near the diagonal $x = y$. 
\[
\varphi(x, \xi) - \varphi(y, \xi) = (x-y)\cdot \int_0^1 \partial_x \varphi(x + t(y-x), \xi)\, dt 
\]
Define the map
\[
 \xi^\prime(x, y, \xi) = \int_0^1 \partial_x\varphi(x + t(y-x), \xi)\, dt.
\]
Near the diagonal, this map is smooth, and
\[
 \det\left(\frac{\partial\xi^\prime}{\partial\xi}(x, x, \xi)\right) = \det\left(\partial_{\xi x}\varphi(x, \xi)\right) = h(x, \xi) > 0.
\]
and so the map $(x, y, \xi) \mapsto (x, y, \xi^\prime)$ is a diffeomorphism of a neighborhood of the diagonal onto another neighborhood of the diagonal. It may be necessary to shrink the support of $\chi$ slightly for the change of coordinates to be well-defined. In these variables,
\[
\varphi(x, \xi) - \varphi(y, \xi) = (x-y)\cdot \xi^\prime.
\]
Both sides are positive homogeneous of degree one, which implies that $|\xi^\prime| = c(x, y)|\xi|$ with $c(x, y)$ a strictly positive, smooth function defined near the diagonal. Clearly $c(x, x) = |\partial_x\varphi(x, \xi)|$. This reduces the cut-off kernel of $N_w^+$ to an honest pseudodifferential operator
\[
\chi K^+_{N_w}\chi = \int_{\R^n} e^{i(x-y)\cdot\xi^\prime} W\left(x, y, \frac{\xi^\prime}{|\xi^\prime|}\right) \chi(x)\chi(y)|\xi^\prime|^{1-n} c(x, y)^{n-1} \left|\det \frac{\partial\xi^\prime}{\partial\xi}\right|^{-1} \, d\xi^\prime
\]
To evaluate the principal symbol of $K^+_{N_w}$, we restrict the amplitude to the diagonal $x = y$. The principal symbol of $N_w$ is the sum of those for $N_w^+$ and $N_w^-$.
\end{proof}

This reconfirms that $N_w$ is an elliptic pseudodifferential operator of order $1-n$, provided the weight is nonvanishing and the global Bolker condition is satisfied.

We now consider the stability of reconstructing $f \in L^2_c(M_1)$ from $R_wf$, using the analysis of the normal operator $N_wf = R_w^* R_wf$ from the previous section. The basic estimate follows from elliptic regularity; the stability estimate follows from \cite[Prop V.3.1]{taylor81}.

\begin{lemma}
Let $w \in \cinf(M \times S^{n-1})$ be a nonvanishing weight and let $\varphi \in \cinf(M \times (\R^n \setminus 0))$ be a defining function. Then for all $f \in L^2(M)$ and $s > 0$ there exists $C > 0$ and $C_s > 0$ depending on $s$ such that
\[
 ||f||_{L^2(M)} \le C||N_w f||_{H^{n-1}(M_1)} + C_s||f||_{H^{-s}}.
\]
If, in addition, $N_w : L^2(M) \to H^{n-1}(M_1)$ is injective, then we have a stability estimate with a loss of $n-1$ derivatives,
\[
 ||f||_{L^2(M)} \le C^\prime||N_w f||_{H^{n-1}(M_1)}
\]
with a different constant $C^\prime > 0$.
\end{lemma}

In particular, by Proposition 1, the latter stability estimate holds when the geometric data (i.e., $M, \varphi$ and $w$) are analytic. This also follows directly from \cite{boman87}. Our main contribution is to extend this stability estimate by perturbation to a generic set of smooth geometric data. We begin by using the standard pseudodifferential calculus to show that the normal operator depends continuously on finitely many derivatives of the data.

\begin{lemma}
Let $(M_1, g)$ be an open Riemannian manifold with an embedded compact manifold $M$ with boundary. Let $\varphi_1, \varphi_2$ be two defining functions and $w_1, w_2$ be two nonvanishing weights. Let $N_1 = R_{w_1}^* R_{w_1}$ and $N_2 = R_{w_2}^* R_{w_2}$. There exists a $K \gg n$ such that if
\[
 ||\varphi_1 - \varphi_2||_{C^K(M_1 \times S^{n-1})}, ||w_1 - w_2||_{C^K(M_1 \times S^{n-1})} < \delta \ll 1,
\]
then there exists $C > 0$ depending \emph{a priori} on the $C^K(M_1 \times S^{n-1})$ norm of $\varphi_1$ and $w_1$ such that
\[
 ||(N_1 - N_2)f||_{H^{n-1}(M_1)} \le C \delta||f||_{L^2(M_1)}.
\]
\end{lemma}

\begin{proof}
We have seen in the previous lemmas that $N_1$ and $N_2$ are both elliptic pseudodifferential operators with symbols depending on $\varphi_1, \varphi_2$ and $w_1, w_2$ respectively. Let $K$ be an arbitrary, large natural number to be fixed later. If the defining functions and weights are $\delta$--close in $C^K(M_1)$, then it follows from Lemma 2 that the amplitudes are $O(\delta)$ in $C^{K-2}(M_1)$. By the continuity of pseudodifferential operators \cite[Theorem 18.3.11 and ff.]{hormander-i} the operator norm of $N^\pm_1 - N^\pm_2$ is bounded by a constant multiplied by some $\cinf(M_1)$--seminorm of the difference of the amplitudes. Take $K$ large enough so that 
\[
 ||N^\pm_1 - N^\pm_2||_{L^2_c(M_1) \to H^{n-1}(M_1)} = O(\delta).
\]
The lemma follows from adding the positive and negative parts of the estimate together. Notice that $K$, the necessary number of derivatives, does not depend on the defining functions or the weights themselves.
\end{proof}

It is of interest to determine the minimal regularity necessary for the perturbation result of the previous lemma. For the related geodesic ray transform, it is known that the geometric data need only be $\delta$--close in $C^2$ \cite{frigyik08}. One would then expect the above to hold for data $\delta$--close in $C^n$.

We now able to prove with our main result, Theorem \ref{stability-perturbation}:

\begin{proof}[{Proof of Theorem \ref{stability-perturbation}}]
Recall $R_w$ is an injective generalized Radon transform. Lemma 3 yields the following stability estimate:
\[
 ||f||_{L^2(M_1)} \le C_1||R_w^* R_w f||_{H^{n-1}(M_1)}.
\]
Then Lemma 4 allows us to perturb this estimate using
\[
 ||(R_w^* R_w - \tilde{R}_{\tilde{w}}^* \tilde{R}_{\tilde{w}}) f||_{H^{n-1}(M_1)} \le C_2\delta||f||_{L^2(M_1)}.
\]
Therefore,
\[
 ||f||_{L^2(M_1)} \le C_1||\tilde{R}_{\tilde{w}}^*\tilde{R}_{\tilde{w}} f||_{H^{n-1}(M_1)} + C_1C_2\delta||f||_{L^2(M_1)}.
\]
For $\delta < \min\{(2C_1C_2)^{-1}, 1/2\}$, the second term on the right-hand side may be absorbed into the left. The resulting stability estimate for the perturbed normal operator implies injectivity of $\tilde{R}_{\tilde{w}}$.
\end{proof}

\bibliographystyle{abbrv}
\bibliography{master}

\end{document}